\documentclass[11pt]{article}

\usepackage{fullpage}
\usepackage{amsmath, graphicx, amsthm, amsfonts, amssymb, amstext}
\usepackage{ragged2e, enumerate, graphicx, lscape, framed}
\usepackage{pgf, tikz, float, subcaption}
\usetikzlibrary{graphs, graphs.standard, decorations.pathreplacing, calligraphy}
\usepackage[symbols,nogroupskip,sort=none]{glossaries-extra}
\usepackage{subfiles}
\usepackage[T1]{fontenc}

\newtheorem{theorem}{Theorem}[section]
\newtheorem{lemma}[theorem]{Lemma}
\newtheorem{proposition}[theorem]{Proposition}

\newtheorem{conjecture}[theorem]{Conjecture}
\newtheorem{problem}[theorem]{Problem}
\newtheorem*{remark}{Remark}

\usepackage{hyperref}
\hypersetup{
	colorlinks=true,
    urlcolor=purple,
	linkcolor=blue,
    citecolor=purple,
}

\DeclareMathOperator{\In}{In}
\DeclareMathOperator{\SRG}{SRG}
\DeclareMathOperator{\rank}{rank}
\DeclareMathOperator{\mult}{mult}

\title{A new conjecture on the inertia of graphs}
\author{Saieed Akbari, Clive Elphick, Hitesh Kumar, Shivaramakrishna Pragada, Quanyu Tang}
\date{}

\begin{document}
\maketitle
\begin{abstract}
Let $G$ be a graph with adjacency matrix $A(G)$. We conjecture that 
\[2n^+(G) \le n^-(G)(n^-(G) + 1),\]
where $n^+(G)$ and $n^-(G)$ denote the number of positive and negative eigenvalues of $A(G)$, respectively. This conjecture generalizes to all graphs the well-known \emph{absolute bound} for strongly regular graphs. The conjecture also relates to a question posed by Torga\v{s}ev. We prove the conjecture for special graph families, including line graphs and planar graphs, and provide examples where the conjecture is exact. We also conjecture that for any connected graph $G$, its line graph $L(G)$ satisfies $n^+(L(G)) \le n^-(L(G)) + 1$, and obtain partial results. 
\end{abstract}



\section{Introduction}

We use standard graph notation and terminology. Throughout the paper, let \( G = (V(G), E(G)) \) be a simple graph of \emph{order} \( n(G) \), \emph{size} \( m(G) \), and \emph{adjacency matrix} \( A(G) \). We denote the eigenvalues of $A(G)$ (also called the \emph{spectrum} of $G$)  by $\lambda_1 \ge \cdots \ge \lambda_n$. Suppose eigenvalue $\lambda_i$ appears with multiplicity $m_i$. Then often we will write the spectrum of $G$ as $(\lambda_1^{(m_1)}, \ldots, \lambda_k^{(m_k)})$, where $\lambda_1>\lambda_2>\cdots > \lambda_k$ are the distinct eigenvalues of $A(G)$. Let $n^+(G), n^0(G)$, and $n^-(G)$ denote the number of positive, zero, and negative eigenvalues of $A(G)$, counted with multiplicities, respectively. The \emph{inertia} of $G$, denoted by $\In(G)$, is the ordered triple $(n^+(G), n^0(G), n^-(G))$. The \emph{rank} of graph $G$ is given by $\rank(G)=n^+(G) + n^-(G)$. The \emph{signature} of a graph $G$ is given by $s(G)=n^+(G) - n^-(G)$. Two vertices $u,v\in V(G)$ are said to be \emph{twins} if they have the same neighbours, i.e., $N(u)=N(v)$. A graph is called \emph{reduced} if $G$ has no twins (or \emph{twin-free}) and has no isolated vertices. For $A\subseteq V(G)$, $G[A]$ denotes the subgraph of $G$ induced by the vertices in $A$, and $G-A$ denotes the subgraph obtained by deleting all vertices in $A$ and all edges incident to a vertex in $A$. When $A = \{u\}$ (resp. $\{u,v\}$), we simply write $G-u$ (resp. $G-u-v$) to denote $G-\{u\}$ (resp. $G-\{u,v\}$). 

The following is a well-known result of Delsarte, Goethals, and Seidel \cite{Seidel_1977} (as discussed in \cite[Section 10]{Brouwer_Haemers_book}), proved using the theory of equiangular lines. 

\begin{theorem}[\cite{Seidel_1977} cf. \cite{Seidel_1979, Brouwer_Haemers_book}]\label{thm:DGS_1977}
    Let $G$ be a regular graph of order $n$ with least eigenvalue $\lambda$. If $\lambda < -1$ and multiplicity of $\lambda$ is $n-k$ for some $k\ge 1$, then $n\le \frac{k(k+1)}{2}-1$.
\end{theorem}

Theorem \ref{thm:DGS_1977} was later generalized by Bell and Rowlinson \cite{Bell_Rowlinson_2003} using the technique of star complements.

\begin{theorem}[\cite{Bell_Rowlinson_2003}]\label{thm:Bell_Rowlinson_generalization}
Let $\lambda$ be an eigenvalue of a graph $G$ with multiplicity $n-k$ for some positive integer $k$. Then either $\lambda\in \{0, -1\}$ or $n\le \frac{k(k+1)}{2}$. Furthermore, if $G$ is regular, then  $n\le \frac{(k-1)(k+2)}{2}$.
\end{theorem}

An important corollary of Theorem \ref{thm:DGS_1977} is the following well-known \emph{absolute bounds} for (primitive) strongly regular graphs (SRGs). These bounds are effectively used to rule out otherwise feasible SRG parameters.
 
\begin{theorem}[Absolute bound for SRGs, cf. \cite{Seidel_1979}]\label{thm:Seidel_absolute_bound}
Let $G = \SRG(n, d, \lambda, \mu)$ denote a strongly regular graph with spectrum $(d^{(1)}, r^{(f)}, s^{(g)})$, where $d>r>0$ and $s<-1$. Then 
   \begin{enumerate}[$(i)$]
       \item $2n \le f(f + 3)$;
       \item $2n \le g(g + 3).$
   \end{enumerate}
\end{theorem}

Neumaier \cite{Neumaier_1981} (cf. \cite[Theorem 11.4.3]{Brouwer_Haemers_book}) generalized the above result to association schemes. 

Note that $n= n^+ + n^-$ for (primitive) SRGs, and thus Theorem \ref{thm:Seidel_absolute_bound}$(ii)$ is equivalent to $2n^+ \le n^-(n^- + 1)$. We believe that this bound is valid for all graphs. 

\begin{conjecture}\label{conj:inertia_main}
For any graph $G$, we have
\[
 2n^+(G) \le n^-(G)(n^-(G) + 1).
\]
\end{conjecture}

The above conjecture, if true, would generalize the absolute bound for SRGs to all graphs. The purpose of this note is to bring the above conjecture to the notice of researchers and provide evidence for its validity. The paper is organized as follows. In Section \ref{section:Torgasev_problem}, we provide more reasons why Conjecture \ref{conj:inertia_main} is interesting. In Section \ref{section:equality_case}, we address the equality case for Conjecture \ref{conj:inertia_main}, and mention some weaker bounds. In Section \ref{section:special_families}, we verify the conjecture for various classes of graphs, including planar graphs and line graphs. We conclude with some remarks in Section \ref{section:conclusion}.

\section{Torga\v{s}ev's problem}
\label{section:Torgasev_problem}

The problem of upper-bounding the order of a graph in terms of some graph invariant involving inertia has received attention in the literature. Note that by adding isolated vertices or twins to a graph, one can increase its order without changing its (positive and negative) inertia (see Lemma \ref{lemma:twin_inertia}). So the problem is meaningful only for reduced graphs.

Motivated by their interest in the connection between the chromatic number and the rank of a graph, the problem of bounding the order of a graph in terms of the rank was first studied by Kotlov and Lov\'{a}sz \cite{Kotlov_Lovasz_1996}. They proved that there is a constant $c>0$ such that any reduced graph with rank $r$ has order at most $c2^{r/2}$. Akbari, Cameron, and Khosrovshahi \cite{Akbari_Cameron_Khosrovshahi} later proposed the following.

\begin{conjecture}[\cite{Akbari_Cameron_Khosrovshahi}]\label{conj:rank_reduced_graphs_order}
For $r \ge 2$, the order of any reduced graph of rank $r$ is at most $m(r)$, where
\[m(r)=
\begin{cases}
    2^{(r+2)/2} -2 \quad \quad \, \text{ if }r \text{ is even},\\
    5\cdot2^{(r-3)/2} - 2 \quad \text{ if }r \text{ is odd}.
\end{cases}
\]
\end{conjecture}

In \cite{Akbari_Cameron_Khosrovshahi}, the authors also provide constructions of graphs for which the bounds in the above conjecture are tight. Ghorbani, Mohammadian, and Tayfeh-Rezaie \cite{Ghorbani_Mohammadian_Behruz_2015} proved that if Conjecture~\ref{conj:rank_reduced_graphs_order} is false, then a counterexample must exist with rank at most $46$. They also showed that every reduced graph of rank $r$ has at most $8m(r) + 14$ vertices.

In the 1980s, Torga\v{s}ev \cite{Torgasev_1985} considered the problem of upper-bounding the order of reduced graphs in terms of $n^-$, and proved a somewhat surprising result. 

\begin{theorem}[Torga\v{s}ev's Theorem \cite{Torgasev_1985}] For any fixed integer $k$, there are finitely many reduced graphs with $n^- = k$.
\end{theorem}

Recently, the above theorem was generalized in \cite{Basso_2021} and \cite{Mohammadian_2022}. 
The analogous statement for $n^+$ is false (complete graphs are easy counterexamples). Torga\v{s}ev \cite{Torgasev_two_neg_1985, Torgasev_three_neg_1989} also determined the maximum order $n(k)$ and maximum positive inertia $n^+(k)$ of reduced graphs with exactly $k$ negative eigenvalues for small values of $k$; see Table~\ref{table:Torgasev}.

 \begin{table}[H]
    \centering
    \begin{tabular}{|c|c|c|}
    \hline
       $n^- = k$  & $n(k)$ & $n^+(k)$\\
    \hline   
       1  & 2 & 1 \\
       2 & 6 & 3\\
       3 & 14 & 6\\
    \hline   
    \end{tabular}
    \caption{Maximum order $n(k)$ and maximum positive inertia $n^+(k)$ of graphs with $n^- = k$.}
    \label{table:Torgasev}
\end{table} 

\begin{remark}
    The only reduced graph with $n^-=1$ and $n^+=1$ is $K_2$. The only reduced graph with $n^-=2$ and $n^+=3$ is the cycle $C_5$. The graphs $H_1$ and $H_2$ given in Figure~\ref{fig:H1_H2} are the only reduced graphs with $n^-=3$ and $n^+=6$. Note that $H_2$ is obtained from $H_1$ by joining a pair of antipodal points on the outer 6-cycle.
\end{remark}

\begin{figure}[!h]
\centering

\begin{subfigure}[b]{0.49\textwidth}
\centering
\begin{tikzpicture}[scale=0.4, main_node/.style={circle,draw,minimum size=0.1em,inner sep=3pt}]
\node[main_node] (0) at (1.4428571428571422, 4.857142857142856) {};
\node[main_node] (1) at (3.5428571428571427, 1.2198361612482114) {};
\node[main_node] (2) at (1.4428571428571433, -2.4174705346464274) {};
\node[main_node] (3) at (-2.7571428571428536, -2.41747053464643) {};
\node[main_node] (4) at (-4.857142857142858, 1.2198361612482147) {};
\node[main_node] (5) at (-2.7571428571428536, 4.857142857142858) {};
\node[main_node] (6) at (-2.7527866908482137, 2.428493998162276) {};
\node[main_node] (7) at (-0.6813581194196421, -1.2000769944437515) {};
\node[main_node] (8) at (1.4614990234375005, 2.428493998162276) {};

 \path[draw, thick]
(1) edge node {} (0) 
(1) edge node {} (2) 
(2) edge node {} (3) 
(3) edge node {} (4) 
(4) edge node {} (5) 
(5) edge node {} (0) 
(7) edge node {} (4) 
(6) edge node {} (7) 
(7) edge node {} (8) 
(7) edge node {} (1) 
(6) edge node {} (3) 
(6) edge node {} (0) 
(6) edge node {} (8) 
(8) edge node {} (5) 
(8) edge node {} (2) 
;
\end{tikzpicture}
\caption{\( H_1 \).}
\end{subfigure}
\hfill
\begin{subfigure}[b]{0.49\textwidth}
\centering
\begin{tikzpicture}[scale=0.4, main_node/.style={circle,draw,minimum size=0.1em,inner sep=3pt}]
\node[main_node] (0) at (1.4428571428571422, 4.857142857142856) {};
\node[main_node] (1) at (3.5428571428571427, 1.2198361612482114) {};
\node[main_node] (2) at (1.4428571428571433, -2.4174705346464274) {};
\node[main_node] (3) at (-2.7571428571428536, -2.41747053464643) {};
\node[main_node] (4) at (-4.857142857142858, 1.2198361612482147) {};
\node[main_node] (5) at (-2.7571428571428536, 4.857142857142858) {};
\node[main_node] (6) at (-2.7527866908482137, 2.428493998162276) {};
\node[main_node] (7) at (-0.6813581194196421, -1.2000769944437515) {};
\node[main_node] (8) at (1.4614990234375005, 2.428493998162276) {};

 \path[draw, thick]
(1) edge node {} (0) 
(1) edge node {} (2) 
(2) edge node {} (3) 
(3) edge node {} (4) 
(4) edge node {} (5) 
(5) edge node {} (0) 
(7) edge node {} (4) 
(6) edge node {} (7) 
(7) edge node {} (8) 
(7) edge node {} (1) 
(6) edge node {} (3) 
(6) edge node {} (0) 
(6) edge node {} (8) 
(8) edge node {} (5) 
(8) edge node {} (2) 
(0) edge node {} (3)
;
\end{tikzpicture}
\caption{\( H_2 \).}
\end{subfigure}
\caption{The graphs \( H_1 \) and \( H_2 \), both satisfying \( n^- = 3 \) and \( n^+ = 6 \).}
\label{fig:H1_H2}
\end{figure}
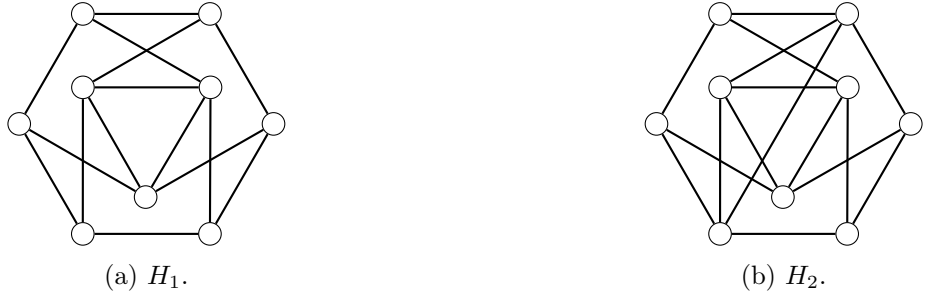

Motivated by the values of $n(k)$, Mohammadian \cite{Mohammadian_2022} proposed the following.

\begin{conjecture}[\cite{Mohammadian_2022}]
    For every non-negative integer $k$, the order of a reduced graph with exactly $k$ negative eigenvalues is at most $2^{k+1}-2$. 
\end{conjecture}

As pointed out in \cite{Mohammadian_2022}, the above conjecture is tight for the following construction of reduced graphs given by Kotlov and Lov\'{a}sz \cite{Kotlov_Lovasz_1996}. Starting with $K_2$ and applying the operation given in the theorem below $k-1$ times gives a graph of order $2^{k+1}-2$ and $n^-=k$.

\begin{theorem}[\cite{Kotlov_Lovasz_1996}]
    Let $G$ be a graph of order $n$, rank $r$, and adjacency matrix $A(G)$. Let $G'$ be the graph on $2n+2$ vertices whose adjacency matrix is given by 
    \[ A(G') = 
    \begin{bmatrix}
    A(G) & A(G) & \mathbf{0} & \mathbf{0}\\
    A(G) & A(G) & \mathbf{1} & \mathbf{0}\\
    \mathbf{0} & \mathbf{1} & \mathbf{0} & \mathbf{1}\\
    \mathbf{0} & \mathbf{0} & \mathbf{1} & \mathbf{0}
    \end{bmatrix},\]
    where $\mathbf{0}$ and $\mathbf{1}$ denote row or column vectors with all zeroes and all ones of appropriate sizes, respectively. Then $n^+(G')=n^+(G)+1$ and $n^-(G')=n^-(G)+1$.  
\end{theorem}

In the final paragraph of \cite{Torgasev_number_of_positive_negative_1992}, Torga\v{s}ev states (paraphrased):
\begin{quote}
... any estimate of the growth of the function $k\rightarrow n^+(k)$ can be of
great importance. By the corresponding results in the papers \cite{Torgasev_two_neg_1985, Torgasev_three_neg_1989}, we
know that $n^+(1)= 1, n^+(2)= 3, n^+(3) = 6$. But, so far, we have no information about this function in the general case.
\end{quote}

But he did not propose any explicit bound on $n^+(k)$. Conjecture \ref{conj:inertia_main} is a step towards Torga\v{s}ev's problem of investigating the function $k\rightarrow n^+(k)$.

In a similar vein, the problem of upper-bounding the order of graphs by the number of non-positive eigenvalues has also been investigated by Charles, Farber, Johnson, Kennedy-Shaffer \cite{Charles_non_positive_2013}. An unpublished conjecture mentioned by Mohar in a lecture series \cite{Mohar_lecture} states that the order of a graph with $k$ non-positive eigenvalues is at most $O(k^2)$. Conjecture \ref{conj:inertia_main}, if true, would imply this. 

It is worth noting that there are results in the literature that upper-bound the positive and negative inertia in terms of other graph parameters; see, for instance, \cite{Ma_Yang_Li_2013, Fan_Wang_2017} and the references therein. An interesting conjecture by Ma, Yang and Li \cite{Ma_Yang_Li_2013} concerns the signature of a graph. 

\begin{conjecture}[\cite{Ma_Yang_Li_2013}]\label{conj:signature_cycles}
Let $G$ be a graph with signature $s(G)$. Then
\[ -c_3(G) \le s(G) \le c_5(G),\]
where $c_3(G)$ and $c_5(G)$ denote the number of cycles in $G$ having length $3$ modulo $4$ and length
$1$ modulo $4$, respectively.
\end{conjecture}

Conjecture \ref{conj:inertia_main} can also be rephrased in terms of signature as follows.

\begin{conjecture}\label{conj:inertia_main_signature_version}
    For any graph $G$, we have 
    \[ s(G) \le \binom{n^-(G)}{2}.\]
\end{conjecture}

\section{A coarse upper bound}
\label{section:equality_case}

We have verified Conjecture \ref{conj:inertia_main} for graphs of order at most 9, and graphs with order at most $100$ in the Wolfram Mathematica database and the \emph{House of Graphs} database. 

We note that the quadratic upper bound for $n^+$ in terms of $n^-$ cannot be improved. The line graph of $K_n$ is called \emph{triangular graph} $T(n)$. It is known that $T(n)$ is an SRG and has spectrum $(2(n-2)^{(1)}, (n-4)^{(n-1)}, (-2)^{(n(n - 3)/2)})$, see \cite[p. 10]{Brouwer_Haemers_book}. The complement of $T(n)$, therefore, has inertia $(n(n - 3)/2 + 1, 0, n-1)$ and so provides an infinite family of graphs for which $n^+ = (n^-)^2-O(n^-)$. 

It is also worth noting that there does exist a rather coarse upper bound for $n^+$ in terms of $n^-$. 

\begin{theorem}[\cite{Akbari_Cameron_Khosrovshahi}]\label{thm:coarse_bound}
For any graph $G$, we have 
\[n^+(G) \leq R(n^-(G) + 2, 2^{n^-(G)}) - n^-(G) - 1,\]
where \( R(m,n) \) denotes the Ramsey number.
\end{theorem}

The upper bound given in the above result is exponential in $n^-$, whereas in Conjecture \ref{conj:inertia_main} the upper bound is quadratic in $n^-$. As the gap is huge, it is worth proving even a polynomial bound in $n^-$, see Problem \ref{problem:polynomial_bound}. 

To the best of our knowledge, the only known examples of reduced graphs for which equality holds in Conjecture~\ref{conj:inertia_main} are the ones listed in Table \ref{table:reduced_tight_examples} below.

\begin{table}[H]
    \centering
    \begin{tabular}{|c|c|c|}
    \hline
    Graph  & Inertia \\
    \hline
    $K_2$ & $(1, 0, 1)$  \\
    $C_5$ & $(3, 0, 2)$ \\
    $H_1$ and $H_2$ (Figure \ref{fig:H1_H2}) & $(6, 0, 3)$\\
    $\SRG(27, 10, 1, 5)$ & $(21, 0, 6)$ \\
    $\SRG(275, 112, 30, 56)$ & $(253, 0,22)$\\
    \hline
    \end{tabular}
    \caption{Reduced graphs which attain the bound in Conjecture \ref{conj:inertia_main}}
    \label{table:reduced_tight_examples}
\end{table}
The $\SRG(27, 10, 1, 5)$ is also known as generalized quadrangle \( GQ(2,4) \), and is the complement of the Schl\"{a}fli graph.  The $\SRG(275, 112, 30, 56)$ is the well-known McLaughlin graph\footnote{See entries on \href{https://mathworld.wolfram.com/GeneralizedQuadrangle.html}{Generalized Quadrangle} and \href{https://mathworld.wolfram.com/McLaughlinGraph.html}{McLaughlin Graph} at Wolfram MathWorld.}. 

In light of the following lemma, infinitely many graphs that attain the bound in Conjecture \ref{conj:inertia_main} can be obtained simply by adding twins to a known tight example. 

\begin{lemma}[\cite{Fan_Qian_2009, Geng_Wu_Wang_2020}]\label{lemma:twin_inertia}
Let $G$ be a graph and $u, v\in V(G)$ be distinct vertices. If
$N(u) = N(v)$, then $n^+(G)=n^+(G-u)$ and $n^-(G)=n^-(G-u)$. 
\end{lemma}

But it seems difficult to find reduced graphs that are tight for Conjecture \ref{conj:inertia_main}. We raise the following question.

\begin{problem} Does there exist an infinite family of reduced graphs for which equality holds in Conjecture \ref{conj:inertia_main}? In fact, are there any examples of reduced graphs which attain equality in Conjecture \ref{conj:inertia_main} apart from the ones mentioned in Table \ref{table:reduced_tight_examples}?
\end{problem}

A particularly intriguing case is when $n^- = 4$. We ask:

\begin{problem}
 Does there exist a (reduced) graph with $n^-=4$ and $n^+ = 10$?
\end{problem}

We point out here that if an $\SRG$ attains the bound in Conjecture \ref{conj:inertia_main}, then it has to be a Smith graph (graphs with Krein parameter $q_{22}^2 = 0$, see \cite[Chapter 1]{Brouwer_Maldeghem_2022} for more). 

\section{Special graph families}
\label{section:special_families}

In this section, we verify Conjecture \ref{conj:inertia_main} for some special graph families. Clearly, the conjecture holds for strongly regular graphs. By the results of Torga\v{s}ev \cite{Torgasev_two_neg_1985, Torgasev_three_neg_1989} (see Table \ref{table:Torgasev}), the conjecture is true for graphs with at most three negative eigenvalues. In what follows, we verify Conjecture \ref{conj:inertia_main} for random graphs, graphs with a cut-vertex (under the assumption that the conjecture holds for 2-connected graphs), subquartic graphs and planar graphs, line graphs, self-complementary graphs, graphs with at most 6 odd cycles, tensor product and join of graphs, and cographs.

\subsection{Random graphs}

Here, we consider the inertia of the random graph $G(n, 1/2)$. Martin and Wong \cite{Martin_Wong_2009} proved that almost all integer matrices have no integer eigenvalues. It follows that $n^0(G(n, 1/2))=0$ with probability tending to 1.

Using Wigner's semicircle law \cite{Wigner_1958} in the form given by Arnold \cite{Arnold_1967} (with appropriate normalization), we have (also see the discussion in \cite[Section 2]{Elphick_Linz_2024})
\[
n^-(G(n, 1/2))= \frac{2n}{\pi} \int_{-1}^0 \sqrt{1 - x^2} dx = \frac{n}{2} \quad \mbox{  and  } \quad n^+(G(n, 1/2)) - 1 = \frac{2n}{\pi} \int_{0}^1 \sqrt{1 - x^2} dx = \frac{n}{2}.\]

The $-1$ term for $n^+$ corresponds to the largest eigenvalue, which is excluded from Wigner's symmetry about zero. Therefore, $n^+ \approx n^-$ and hence Conjecture \ref{conj:inertia_main} holds for almost all graphs.

\subsection{Graphs with a cut vertex}

If \( G \) is the disjoint union of two graphs \( G_1 \) and \( G_2 \), and both \( G_1 \) and \( G_2 \) satisfy Conjecture~\ref{conj:inertia_main_signature_version} (equivalently, Conjecture~\ref{conj:inertia_main}), then 
\begin{equation}\label{eq:disconnected_inertia}
   s(G) = s(G_1)+s(G_2)\le \binom{n^-(G_1)}{2} + \binom{n^-(G_2)}{2}\le \binom{n^-(G_1)+n^-(G_2)}{2}=\binom{n^-(G)}{2}. 
\end{equation}

Hence, to prove Conjecture \ref{conj:inertia_main_signature_version} (equivalently, Conjecture \ref{conj:inertia_main}), it suffices to work with connected graphs. We now prove that it suffices to work with 2-connected graphs. We first recall some known results. We call a vertex a \emph{leaf} if its degree is $1$.

\begin{lemma}[\cite{Ma_Yang_Li_2013}]\label{lemma:inertia_leaf}
    Let $G$ be a graph containing a leaf $u$ with neighbour $v$, and $H=G-u-v$. Then $n^+(G)=n^+(H)+1$ and $n^-(G)=n^-(H)+1$. 
\end{lemma}

\begin{lemma}[\cite{Ma_Wong_Tian_2016}]\label{lemma:rank_vertex_deleted_subgraph}
Let $G$ be a graph, $v \in V(G)$ and $H = G - v$. Then 
\[|s(G) - s(H)| \leq 1.\]
Moreover, if $\rank(G) = \rank(H)$ or $\rank(G) = \rank(H) + 2$, then $s(G) = s(H).$
\end{lemma}

\begin{theorem}\label{thm:2_connected}
For any  graph $G$, we have
\[
s(G) \leq \binom{n^{-}(G)}{2},
\]
provided the inequality holds for all $2$-connected graphs.
\end{theorem}

\begin{proof}
We proceed by induction on the order $n$. When $n\le 3$, the assertion is trivial. So let $G$ be a graph of order $n\ge 4$. 
In light of \eqref{eq:disconnected_inertia}, we can assume that $G$ is a connected graph. If $G$ is 2-connected, then the assertion holds by assumption. So assume $G$ is a connected graph, which is not 2-connected.

First, suppose that $G$ has a leaf $u$ with neighbour $w$. By Lemma \ref{lemma:inertia_leaf}, we have $s(G)=s(G-u-w)$. By the induction hypothesis, we conclude 
\[ s(G) = s(G-u-w) \le \binom{n^-(G-u-w)}{2} \le \binom{n^-(G)}{2}.\]
So, in what follows, assume that the minimum degree $\delta(G)\ge 2$. By the choice of $G$, it has a cut-vertex, say $v$. Let $H=G-v$. By Lemma \ref{lemma:rank_vertex_deleted_subgraph}, $|s(G)-s(H)|\le 1$, and so one of the following holds:
\begin{enumerate}[$(a)$]
    \item $n^+(G) = n^+(H)$ and $n^-(G)=n^-(H)$,
    \item $n^+(G) = n^+(H) + 1$ and $n^-(G)=n^-(H)$,
    \item $n^+(G) = n^+(H)$ and $n^-(G)=n^-(H)+1$,
    \item $n^+(G) = n^+(H)+1$ and $n^-(G)=n^-(H)+1$.
\end{enumerate}

In cases $(a), (c)$ and $(d)$, we see that $s(G)\le s(H)$. So by the induction hypothesis, we have
\[ s(G) \le s(H) \le \binom{n^-(H)}{2} \le \binom{n^-(G)}{2}.\]

Now, we consider the case $(b)$. In this case, $s(G) = s(H) + 1$. Denote the components of $H$ by $C_1, \dots, C_t$, where $t \geq 2$ by the definition of $H$. We have
\[
n^{+}(H) = \sum_{i=1}^t n^{+}(C_i), \quad n^{-}(H) = \sum_{i=1}^t n^{-}(C_i), \quad s(H) = \sum_{i=1}^t s(C_i).
\]
By the induction hypothesis, we have
\[
s(C_i) \leq \binom{n^{-}(C_i)}{2},\] 
for each $i$, which implies 
\[s(H) \leq \sum_{i=1}^t \binom{n^{-}(C_i)}{2} = \binom{n^{-}(H)}{2} - \sum_{i<j} n^{-}(C_i)n^{-}(C_j).
\]
Since $\delta(G)\ge 2$, each $C_i$ has order at least 2, implying $n^{-}(C_i) \geq 1$, for $i=1, \ldots, t$. Hence
\[
\sum_{i<j} n^{-}(C_i)n^{-}(C_j) \geq 1.
\]
It follows that
\[
s(G) = s(H) + 1 \leq \binom{n^{-}(H)}{2}  =  \binom{n^{-}(G)}{2}.
\]
This completes the proof.
\end{proof}

\subsection{Subquartic graphs and planar graphs}

In this subsection, we verify Conjecture \ref{conj:inertia_main} for subquartic graphs (i.e., graphs with maximum degree at most 4) and planar graphs.

Let $G$ be a graph with chromatic number $\chi(G)\le 4$. We can assume that $n^-(G)\ge 4$; otherwise Conjecture \ref{conj:inertia_main} holds by Table \ref{table:Torgasev}. 

Recently, the following result was obtained by Elphick, Tang and Zhang \cite{Elphick_Tang_Zhang_2025}.

\begin{theorem}[\cite{Elphick_Tang_Zhang_2025}]\label{thm:inertia_fractional_chromatic}
For any graph $G$, we have
\[
1 + \max{\left\{\frac{n^+(G)}{n^-(G)} , \frac{n^-(G)}{n^+(G)}\right\}} \le \chi_f(G),
\]
where $\chi_f(G)$ denotes the fractional chromatic number of $G$. 
\end{theorem}

An immediate corollary is that 
\[
1 + \max{\left\{\frac{n^+(G)}{n^-(G)} , \frac{n^-(G)}{n^+(G)}\right\}} \le \chi(G).
\]

So if $G$ is a graph such that $\chi(G)\le 3$, then 
\[n^+(G)\le 2n^-(G),\] which implies $2n^+(G)\le n^-(G)(n^-(G)+1)$, whenever $n^-(G)\ge 3$. Hence, Conjecture \ref{conj:inertia_main} holds for $G$. In particular, we note that Conjecture \ref{conj:inertia_main} holds for subcubic graphs.

Now, assume that $\chi(G)=4$. Arguing as before, we have 
\[n^+(G)\le 3n^-(G).\]
So if $n^-(G)\ge 5$, then Conjecture \ref{conj:inertia_main} holds for $G$. The remaining case is $n^-(G)=4$. We leave this case for future work, but we prove Conjecture \ref{conj:inertia_main} for two families: graphs with maximum degree 4 and planar graphs. We require the following lemmas.

\begin{lemma}[\cite{Gregory_Heyink_Meulen_2003}]\label{lemma:inertia-interlacing}
Let \( A \) and \( B \) be two Hermitian matrices of the same order. Then\begin{enumerate}[$(i)$]
    \item $n^+(A) - n^-(B)\le n^+(A+B) \le n^+(A) + n^+(B)$.
    \item $n^-(A) - n^+(B) \le n^-(A+B) \le n^-(A) + n^-(B)$.
\end{enumerate}
\end{lemma}

\begin{lemma}\label{lemma:inertia_neighbourhood_deletion}
Let $G$ be a connected graph of order $n\ge 2$, $u\in V(G)$ and $H=G-N[u]$. Then $n^+(G)\ge n^+(H)+1$ and $n^-(G)\ge n^-(H)+1$. 
\end{lemma}

\begin{proof}
Let $w\in N(u)$. Consider the induced subgraph $H'=G-(N(u)\backslash \{w\})$ of $G$. Clearly, $u$ is a leaf with neighbour $w$ in $H'$ and $H'-\{u,w\} = H$. By Lemma \ref{lemma:inertia_leaf}, we have 
\[ n^+(H')\ge n^+(H)+1 \quad \text{and} \quad n^-(H')\ge n^-(H)+1.\]
By the Interlacing Theorem, we have $n^+(G)\ge n^+(H')$ and $n^-(G)\ge n^-(H')$. The proof is complete.
\end{proof}

\begin{theorem}\label{thm:max_degree_4}
Conjecture \ref{conj:inertia_main} holds for graphs with maximum degree four.
\end{theorem}

\begin{proof}
Let $G$ be a graph with maximum degree $\Delta(G)\le 4$. By Brooks' Theorem and the discussion above, we only need to consider the case $n^-(G)=4$. Let $u\in V(G)$, and $H=G-N[u]$. By Lemma \ref{lemma:inertia_neighbourhood_deletion}, we have $n^-(H)\le 3$. By Table \ref{table:Torgasev}, we get $n^+(H)\le 6$. Define $F$ to be the subgraph of $G$ induced by the edges $E(G)\backslash E(H)$. Clearly, $N[u]\subseteq V(F)$. Observe that $E(F)$ can be decomposed into at most four stars centred at the vertices in $N(u)$. By Lemma \ref{lemma:inertia-interlacing}, we see that 
\[ n^+(G)\le n^+(H) + n^+(F)\le 6 + 4 = 10.\]
Since, $n^-(G)=4$, we have $n^-(G)\left(n^-(G)+1\right)=20$, and Conjecture \ref{conj:inertia_main} holds for $G$.
\end{proof}

Next, we verify Conjecture \ref{conj:inertia_main} for planar graphs. We require the following fact.

\begin{proposition}\label{prop:planar_negative_3}
    If $G$ is a planar graph with $n^-(G)\le 3$, then $n^+(G)\le 5$.
\end{proposition}

\begin{proof}
Let $G$ be a planar graph with $n^-(G)\le 3$. We can further assume that $G$ is reduced. By a result of Torga\v{s}ev \cite[Corollary 1]{Torgasev_three_neg_1989}, $G$ is an induced subgraph of one of the 32 graphs in Table 2 of \cite{Torgasev_three_neg_1989}. These graphs have orders in $\{9, 10, 11, 12, 13, 14\}$. We verified that all the graphs of order at least 10 in the table have $n^+\le 5$, and so their induced subgraphs also have $n^+\le 5$ by the Interlacing Theorem. We verified by computer\footnote{We used the database \href{https://users.cecs.anu.edu.au/~bdm/data/graphs.html}{here} to obtain all planar graphs of order up to 9.} that all planar graphs of order at most 9 with \( n^- \leq 3 \) satisfy \( n^+ \leq 5 \). The assertion follows.
\end{proof}

\begin{theorem}
Conjecture \ref{conj:inertia_main} holds for planar graphs.
\end{theorem}

\begin{proof}
 By the Four Color Theorem and the discussion above, we only need to consider a planar graph $G$ with $n^-(G)=4$. Since $G$ is planar, there exists $u\in V(G)$ such that $\deg(u)\le 5$. Let $H=G-N[u]$. By Lemma \ref{lemma:inertia_neighbourhood_deletion}, we have $n^-(H)\le 3$. By Proposition \ref{prop:planar_negative_3}, $n^+(H)\le 5$. As in the proof of Theorem \ref{thm:max_degree_4}, let $F$ be the subgraph of $G$ induced by the edges $E(G)\backslash E(H)$. Clearly, $E(F)$ can be decomposed into at most five stars centred at the vertices in $N(u)$. Now, by Lemma \ref{lemma:inertia-interlacing}, we see that 
 \[ n^+(G)\le n^+(H) + n^+(F)\le 5 + 5 = 10.\]
As $n^-(G)\left(n^-(G)+1\right)=20$, Conjecture \ref{conj:inertia_main} holds for $G$.
\end{proof}

\subsection{Line graphs}

Here, we verify Conjecture \ref{conj:inertia_main} for line graphs.

Let $G$ be a graph of order $n$, size $m$, and line graph $L(G)$. Let $\mathcal{L}(G)$ and $\mathcal{Q}(G)$ denote its Laplacian and signless Laplacian matrix. It is well-known that if $\theta_1\ge \cdots \ge \theta_n$ are the eigenvalues of $\mathcal{Q}(G)$, then $\theta_i - 2$ are the eigenvalues of $A(L(G))$ with remaining eigenvalues equal to $-2$, refer \cite[Chapter 1]{Brouwer_Haemers_book}. In particular, the least eigenvalue $\lambda_{\min}(L(G))\ge -2$ and has multiplicity at least $m-n$.

We first prove the following lemma. Recall that for a graph $G$ of order $n$, the \emph{energy} of $G$ is defined as $\mathcal{E}(G)=\sum_{i=1}^n |\lambda_i(G)|$. 

\begin{lemma}\label{lemma:inertia_least_eigenvalue}
For any graph $G$ of order $n$, we have
    \[ n^+(G) \leq n^-(G)\left(2|\lambda_n(G)|-1\right).\]
Moreover, if $\lambda_1(G) \geq 3.3$, then
\[n^+(G) \leq n^-(G)\left(2|\lambda_n(G)|-1\right) - 1.1.\]    
\end{lemma}

\begin{proof}
Observe that for any $x\in \mathbb{R}^+$ we have $x \geq \ln x +1$. Now, 
\[
    \mathcal{E}(G) = \sum_{i=1}^n |\lambda_i(G)| \geq n^+(G)  + n^-(G) + \ln \left(\prod_{\lambda_i(G)\neq 0} |\lambda_i(G)|\right).
\]
Note that $\prod_{\lambda_i(G)\neq 0} |\lambda_i(G)|$ is the first non-zero coefficient of the characteristic polynomial of $G$. Hence $\prod_{\lambda_i(G)\neq 0} |\lambda_i(G)| \geq 1$ and therefore 
\[\mathcal{E}(G) \geq n^+(G)  + n^-(G).\]
Since the trace of the adjacency matrix $A(G)$ is zero, we know
\[
    \mathcal{E}(G) = 2\sum_{\lambda_i(G) < 0}|\lambda_i(G)| \leq 2|\lambda_{n}(G)|n^-(G).
\]
Combining the above inequalities for $\mathcal{E}(G)$ gives the first inequality in the assertion. 

Now, if $x\in \mathbb{R}$ is such that $x \geq 3.3$, then $x \geq \ln x + 2.1$. So if $\lambda_1(G) \ge 3.3$, then following the same steps as above gives the second inequality in the assertion.
\end{proof}

We note that the inequality $\mathcal{E}(G) \geq n^+(G)  + n^-(G)$ was first observed in \cite{fajtlowicz1987} (cf. \cite{Akbari_Ghorbani_Zare_2009}).

\begin{theorem}\label{thm:line_graphs}
For any graph $G$ with line graph $L(G)$, we have 
    \[ n^+(L(G))\le \min \left\{3n^-(L(G)),\, \frac{1}{2}n^-(L(G))\left(n^-(L(G))+1\right)\right\}.\]
\end{theorem}

\begin{proof} Let $G$ be a connected graph of order $n$. If $n\le 8$, then the assertion is verified using a computer. So we assume $n\ge 9$. Throughout this proof, we denote $n^+(L(G))$ and $n^-(L(G))$ by $n^+$ and $n^-$, respectively. 

Using the first inequality in Lemma \ref{lemma:inertia_least_eigenvalue}, we get $n^+\le 3n^-$ since $|\lambda_{\min}(L(G))|\le 2$.  If $n^-\ge 5$, then $3n^- \le \frac{1}{2}n^-(n^-+1)$ and the assertion follows. If $n^-\le 3$, we are done by Table \ref{table:Torgasev}. 

So, assume that $n^- = 4$. We need to show that $n^+\le 10$. We consider the following cases:

\textbf{Case 1:} $\lambda_1(L(G)) \geq 3.3$. ~ Then using the second inequality in Lemma \ref{lemma:inertia_least_eigenvalue} and noting that $n^+$ is an integer, we get $n^+ \leq 3n^- - 2 = 10.$ 

\textbf{Case 2:} $\lambda_1(L(G)) < 3.3$. ~ Since $K_{\Delta(G)}$ is an induced subgraph of $L(G)$, we have $\lambda_1(L(G))\ge \Delta(G)-1$ by the Interlacing Theorem. It follows that $\Delta(G) \leq 4$. 

Since $n^- = 4$ and the least eigenvalue $\lambda_{\min}(L(G))$ has multiplicity at least $m-n$, we have $m\le n^- + n = 4 + n$. If $\deg(u)\ge 3$ for all $u\in V(G)$, then 
\[n+4 \ge m = \frac{1}{2}\sum_{u\in V(G)}\deg(u)\ge \frac{3n}{2},\] 
implying $n\le 8$, a contradiction. So let $u\in V(G)$ be such that $\deg(u)\le 2$ and let $v\in N(u)$. Then $\deg(u)+\deg(v)\le 2 + 4 = 6$. It follows that $\deg_{L(G)}(e)\le 4$, where $e=uv$. Let $H=L(G) - N_{L(G)}[e]$. Using Lemma \ref{lemma:inertia_neighbourhood_deletion}, we get $n^-(H)\le 3$. By Table \ref{table:Torgasev}, we get $n^+(H)\le 6$. As in the proof of Theorem \ref{thm:max_degree_4}, using Lemma \ref{lemma:inertia-interlacing}, we get
\[ n^+\le n^+(H) + 4 = 10. \qedhere\]  
\end{proof}

Computational investigation suggests that Theorem \ref{thm:line_graphs} can be further improved.

\begin{conjecture}\label{conj:line_graph}
For any connected graph \( G \), with line graph \( L(G) \),
    \[
    n^+(L(G)) \leq n^-(L(G)) + 1.
    \]
Equivalently, $s(L(G))\le 1$.    
\end{conjecture}

We have tested this conjecture on numerous graphs with $m(G) = n(L(G)) \le 100$ using the LineGraph function in Wolfram Mathematica and found no counterexample. In addition, no counterexample was found for any graph \( G \) with at most 9 vertices. Conjecture \ref{conj:line_graph} is tight, for example, for odd cycle $C_{4k+1}$ whenever $k\ge 1$, Kayak Paddle graph\footnote{See \href{https://mathworld.wolfram.com/KayakPaddleGraph.html}{KayakPaddle} at Wolfram MathWorld} $K(4,5,1)$ and $K(5,5,1)$, and the 5-prism shown in Figure \ref{fig:5_prism}. We require that the line graph be connected; otherwise, $C_5 \cup C_5$ is a disconnected counterexample. 

\begin{figure}
    \centering
 \begin{tikzpicture}[scale = 0.65]
\fill[line width=1pt, fill opacity = 0] (-1,0) -- (1,0) -- (1.6180339887498947,1.9021130325903064) -- (0,3.0776835371752522) -- (-1.6180339887498945,1.9021130325903073) -- cycle;
\fill[line width=1pt, fill opacity = 0] (-1.8,-1) -- (1.8,-1.02) -- (2.9314823100757135,2.3976231187750523) -- (0.030776835371753394,4.529830366915455) -- (-2.893440049423907,2.4299837985500523) -- cycle;
\draw [line width=1pt] 
(-1,0)-- (1,0)
(1,0)--(1.6180339887498947,1.9021130325903064)
(1.6180339887498947,1.9021130325903064)-- (0,3.0776835371752522)
(0,3.0776835371752522)-- (-1.6180339887498945,1.9021130325903073)
(-1.6180339887498945,1.9021130325903073)-- (-1,0)
(-1.8,-1)-- (1.8,-1.02)
(1.8,-1.02)-- (2.9314823100757135,2.3976231187750523)
(2.9314823100757135,2.3976231187750523)-- (0.030776835371753394,4.529830366915455)
(0.030776835371753394,4.529830366915455)-- (-2.893440049423907,2.4299837985500523)(-2.893440049423907,2.4299837985500523)-- (-1.8,-1)
(0,3.0776835371752522)-- (0.030776835371753394,4.529830366915455)
 (1.6180339887498947,1.9021130325903064)-- (2.9314823100757135,2.3976231187750523)
 (1,0)-- (1.8,-1.02)
 (-1,0)-- (-1.8,-1)
(-1.6180339887498945,1.9021130325903073)-- (-2.893440049423907,2.4299837985500523);

\draw[fill=white]
(-1,0) circle (4pt)
(1,0) circle (4pt)
(1.6180339887498947,1.9021130325903064) circle (4pt)
(0,3.0776835371752522) circle (4pt)
(-1.6180339887498945,1.9021130325903073) circle (4pt)
(-1.8,-1) circle (4pt)
(1.8,-1.02) circle (4pt)
(2.9314823100757135,2.3976231187750523) circle (4pt)
(0.030776835371753394,4.529830366915455) circle (4pt)
(-2.893440049423907,2.4299837985500523) circle (4pt);
\end{tikzpicture}
    \caption{5-prism}
    \label{fig:5_prism}
\end{figure}

It is known that for bipartite graphs, the signless Laplacian and the Laplacian spectrum are the same; see, for instance, \cite[Chapter 1]{Brouwer_Haemers_book}. The following was shown by Guo \cite{Guo_2007} (cf.~\cite{Braga_Rodrigues_Trevisan_2013}).

\begin{theorem}[\cite{Guo_2007}, cf.~\cite{Braga_Rodrigues_Trevisan_2013}]\label{thm:tree_Laplacian}
Any tree of order $n$ has at least $\lceil \frac{n}{2}\rceil$ Laplacian eigenvalues in the interval $[0,2)$.
\end{theorem}

Using the above facts, it is easy to see that for any tree $T$, we have $s(L(T))\le -1$. It follows that Conjecture \ref{conj:line_graph} holds for line graphs of trees. 

Furthermore, we note that Conjecture \ref{conj:line_graph} is true for dense line graphs. In other words, if $G$ has size $m$ and order $n$, with $m\ge 2n-1$, then 
    \[n^+(L(G))\le n \le m-n+1 \le n^-(L(G))+1.\]
So Conjecture \ref{conj:line_graph} is open only for the line graph of $G$ when $n\le m\le 2n-2$.

\subsection{Self-complementary graphs}

A graph \( G \) is called \emph{self-complementary} if it is isomorphic to its complement $\overline{G}$.

\begin{theorem}\label{thm:self-complementary-graphs1}
Let \( G \) be a self-complementary graph. Then 
\[
n^+(G) \le n^-(G) + 1.
\]
\end{theorem}

\begin{proof}
It is known that for any graph \( G \) of order \( n \), the following hold (see \cite[Theorem 7]{Elphick_Wocjan_2017}):
\[
n^+(G) + n^+(\overline{G}) \le n + 1 \quad \text{and} \quad n - 1 \le n^-(G) + n^-(\overline{G}).
\]
For a self-complementary graph $G$, these inequalities simplify to 
\[
2n^+(G) \le n + 1 \quad \text{and} \quad 2n^-(G) \ge n - 1.
\]
Thus,
\[
n^+(G) \le \frac{n + 1}{2} \le n^-(G) + 1.
\qedhere\] \end{proof}

Any self-complementary graph of order $n\ge 6$ contains a triangle since the Ramsey number $R(3,3)=6$. Hence by the Interlacing Theorem, $n^-(G)\ge n^-(K_3)=2$. Moreover, there is no self-complementary graph of order at most 5 with $n^-=1$. Thus, by Theorem~\ref{thm:self-complementary-graphs1}, we have 
\[
2n^+(G) \le n^-(G)(n^-(G) + 1),
\]
i.e., any self-complementary graph \( G \) satisfies Conjecture~\ref{conj:inertia_main}.

\begin{remark}
The inequality in Theorem \ref{thm:self-complementary-graphs1} is tight for Paley graphs\footnote{See \href{https://mathworld.wolfram.com/PaleyGraph.html}{Paley graphs} at Wolfram MathWorld.}.
\end{remark}

\subsection{Graphs with few odd cycles}

In this subsection, we verify Conjecture \ref{conj:inertia_main} for graphs with at most 6 odd cycles. In \cite[Theorem 5.2]{Ma_Yang_Li_2013}, the following inequality was established for a general graph \( G \).

\begin{theorem}[{\cite{Ma_Yang_Li_2013}}]\label{thm:signature-and-odd-cycle}
Let \( G \) be a graph. Then 
\[
|n^+(G) - n^-(G)| \le c_1(G),
\]
where \( c_1(G) \) denotes the number of odd cycles in \( G \).
\end{theorem}

Therefore, if \(c_1(G) \leq \binom{n^-(G)}{2}\), then Conjecture~\ref{conj:inertia_main} holds. By Table \ref{table:Torgasev}, it suffices to consider the case \( n^-(G) \geq 4 \). So, if \( c_1(G) \leq \binom{4}{2} = 6 \), Conjecture~\ref{conj:inertia_main} holds. In particular, the conjecture holds for unicyclic, bicyclic, and tricyclic graphs.

\subsection{Graph products}

We verify Conjecture \ref{conj:inertia_main} for some graph products, namely the tensor product and the join of two graphs. For the definitions of these graph products, refer \cite{Barik_op_spec_2018}.

\begin{proposition}
Let $G$ and $H$ be two graphs. Then 
\[ 2n^+(G\otimes H)\le n^-(G\otimes H)(n^-(G\otimes H)+1),\]
where $G\otimes H$ denotes the direct (tensor) product of $G$ and $H$.
\end{proposition}
  
\begin{proof}
It is known that if $\lambda_i$ ($1\le i\le |V(G)|$) and $\mu_j$ ($1\le j\le |V(H)|$) are the eigenvalues of $G$ and $H$ respectively, then the eigenvalues of $G\otimes H$ are given by $\lambda_i\mu_j$ (see \cite{Barik_op_spec_2018}). It follows that   
\[ n^+(G\otimes H) = n^+(G)n^+(H) + n^-(G)n^-(H); \quad n^-(G\otimes H) = n^+(G)n^-(H) + n^-(G)n^+(H).\]
We have 
\begin{align*}
     n^-(G\otimes H)(n^-(G\otimes H)+1) & = (n^+(G)n^-(H) + n^-(G)n^+(H))^2 + n^-(G\otimes H)\\
     & \ge 4n^+(G)n^-(G)n^+(H)n^-(H) +  n^-(G\otimes H)\\
     & \ge 2n^+(G)n^+(H) + 2n^-(G)n^-(H) + n^-(G\otimes H)\\
     & > 2n^+(G\otimes H).\qedhere
\end{align*}
\end{proof}

Next, we consider the join of graphs. We require the following lemma.

\begin{lemma}\label{lemma:twins_negative_inertia}
    Let $G$ be a graph and $u,v\in V(G)$ be such that $N[u]=N[v]$. Then
    \[ n^-(G )\ge n^-(G -u-v) +1.\]
\end{lemma}

\begin{proof} Let $n=|V(G)|$ and $H=G-u-v$. Let $k=n^-(H)$ and $x_1, \ldots, x_k$ denote the orthogonal unit eigenvectors corresponding to the negative eigenvalues of $A(H)$. Let $x_{k+1} \in \mathbb{R}^2$ be the unit eigenvector of the adjacency matrix of $K_2\cong G[\{u,v\}]$ corresponding to the eigenvalue $-1$, given by
\[
x_{k+1} = \frac{1}{\sqrt{2}} \begin{bmatrix} 1 \\ -1 \end{bmatrix}.
\]
Extend $x_i$ to $\tilde{x}_i \in \mathbb{R}^n$ by padding with zeroes for $1 \le i \le k+1$. Consider the $(k+1)$-dimensional subspace $W$  of $\mathbb{R}^n$ spanned by the orthogonal unit vectors $\tilde{x}_1, \cdots, \tilde{x}_{k+1}$ and $z=\alpha_1\tilde{x}_1 + \cdots + \alpha_{k+1}\tilde{x}_{k+1}\in W$ with $\sum_{i=1}^{k+1}\alpha_i^2 = 1$. Using the Courant-Fischer-Weyl Min-Max Theorem (see \cite{Horn_Johnson_2013}), we have
\begin{align*}
   \lambda_{n-k-1}(G) & \le \max_{z} z^T A(G)z \\
   & = (\alpha_1\tilde{x}_1 + \cdots + \alpha_{k+1}\tilde{x}_{k+1})^T A(G) (\alpha_1\tilde{x}_1 + \cdots + \alpha_{k+1}\tilde{x}_{k+1})\\
   & = \alpha_{k+1}^2x_{k+1}^TA(K_2)x_{k+1} + \sum_{i=1}^{k} \alpha_i^2 x_i^TA(H)x_i\\
   & < 0.
\end{align*}
We conclude that $n^-(G)\ge n^-(H)+1$. 
\end{proof}

\begin{theorem}
Let $G$ and $H$ be two connected graphs, each of order at least $2$, for which Conjecture~\ref{conj:inertia_main} holds. Then 
\[ 2n^+(G\vee H)\le n^-(G\vee H)(n^-(G\vee H)+1),\]
where $G\vee H$ denotes the join of $G$ and $H$.
\end{theorem}

\begin{proof}
Without loss of generality, we can assume that $n^-(G)\ge n^-(H)\ge 1$. By Lemma \ref{lemma:inertia-interlacing}, we have 
\[n^+(G\vee H)\le n^+(G)+n^+(H)+n^+(K_{|G|,|H|})= n^+(G)+n^+(H)+1,\]
and 
\[ n^-(G\vee H)\ge n^-(G)+n^-(H)-n^-(K_{|G|, |H|})= n^-(G)+n^-(H) -1.\]
We consider the following cases:

\textbf{Case 1:} $n^-(H)\ge 2$. ~ If $n^-(G)\ge 3$, then
\begin{align*}
   2n^+(G\vee H) & \le 2n^+(G)+2n^+(H)+2\\
   & \le n^-(G)\left(n^-(G)+1\right) + n^-(H)\left(n^-(H)+1\right) + 2\\
   & \le \left(n^-(G)+n^-(H)-1\right) \left(n^-(G)+n^-(H)\right)\\
   & \le n^-(G\vee H)\left( n^-(G\vee H)  +1 \right).   
\end{align*}

And if $n^-(G)=2$, then $n^-(H)=2$. By Table \ref{table:Torgasev}, we have $\max\{n^+(G), n^+(H)\}\le 3$. It follows that 
\[ n^+(G\vee H)\le 3 + 3 + 1 =7 \ \text{ and }\ n^-(G\vee H)\ge 2 +2 - 1 = 3.\]
If $n^-(G\vee H)=3$, then we are done by Table \ref{table:Torgasev}. So assume $n^-(G\vee H)\ge 4$, which implies 
\[ 2n^+(G\vee H)\le 14 < n^-(G\vee H)\left(n^-(G\vee H) + 1\right).\]

\textbf{Case 2:} $n^-(H)=1$. ~ Then $H$ is a complete bipartite graph, say $K_{p,q}$ with $p+q\ge 2$. 

Again, by Lemma \ref{lemma:inertia-interlacing}, we have 
\begin{align*}
   2n^+(G\vee H)& \le 2n^+(G) + 2n^+(K_{p,q,|G|})\\
   & =2n^+(G) +2\\
   & \le n^-(G)\left(n^-(G)+1\right)  + 2\\
   & \le \left(n^-(G)+1\right)\left(n^-(G)+2\right)\\
   & \le n^-(G\vee K_2)\left(n^-(G\vee K_2)  +1 \right)\\
   & \le n^-(G\vee H)\left(n^-(G\vee H)  +1 \right).
\end{align*}
The second last inequality holds by Lemma \ref{lemma:twins_negative_inertia} and the last inequality holds by the Interlacing Theorem since $G\vee K_2$ is an induced subgraph of $G\vee H$. This completes the proof.
\end{proof}

\subsection{Cographs}

A graph is called a \emph{cograph} if it does not contain the 4-vertex path $P_4$ as an induced subgraph. Another characterization is that every nontrivial induced subgraph of a cograph has a pair of vertices with the same open or closed neighbourhoods. We verify Conjecture \ref{conj:inertia_main} for cographs. We first recall some results.

\begin{theorem}[\cite{Mohammadian_Trevisan_2016, Ghorbani_2019}]\label{thm:cograph_eigenvalues}
    Let $G$ be a cograph, $u,v\in V(G)$ be distinct vertices, and $H=G-u$. Let $\mult(G,\lambda)$ denote the eigenvalue multiplicity of $\lambda$ in the spectrum of $G$. Then the following statements are true.
    \begin{enumerate}[$(i)$]
        \item $G$ has no eigenvalues in the interval $(-1,0)$.
        \item If $N(u)=N(v)$, then $\mult(G, 0)=\mult(H,0)+1$.
        \item If $N[u]=N[v]$, then $\mult(G,-1) = \mult(H, -1)+1$.
    \end{enumerate}
\end{theorem}

\begin{theorem}\label{thm:cograph}
    Let $G$ be a cograph. Then 
    \[ n^+(G)\le n^-(G).\]    
\end{theorem}

\begin{proof}
We proceed by induction on the order $n$. If $n=3$, then the assertion holds. Let $G$ be a cograph on $n\ge 4$ vertices. Then $G$ has a pair of vertices $u,v$ such that one of the following occurs:

   \textbf{Case 1:} $u$ and $v$ have the same open neighbourhood, i.e., $N(u)=N(v)$.
   
   Let $H=G-u$. By Lemma \ref{lemma:twin_inertia}, $n^+(H)=n^+(G)$ and $n^-(G)=n^-(H)$. The assertion holds by the induction hypothesis.

   \textbf{Case 2:}  $u$ and $v$ have the same closed neighbourhood, i.e., $N[u]=N[v]$. 
  
   By the Interlacing Theorem, if $H$ has $t$-many eigenvalues less than $-1$, then $G$ has at least $t$-many eigenvalues less than $-1$. Using Theorem \ref{thm:cograph_eigenvalues} $(i)$ and $(iii)$, we conclude that $n^-(G)=n^-(H)+1$.
 
   By the induction hypothesis, we have 
   \[ n^+(G)\le n^+(H)+1 \le n^-(H)+1 = n^-(G).\]
   The proof is complete.
\end{proof}

\section{Concluding remarks}
\label{section:conclusion}

The purpose of this article is to introduce and motivate Conjectures \ref{conj:inertia_main} and \ref{conj:line_graph}. We have made modest progress in proving these conjectures, but the evidence for them appears to be strong, and there are diverse graphs for which the conjectures are tight. We hope that this paper will encourage others to make further progress. Also, since these conjectures do not involve NP-hard parameters, they are well-suited to the use of AI tools to search for counterexamples. 

Given the apparent difficulty of proving Conjecture \ref{conj:inertia_main}, we believe the following weaker conjecture may be more tractable.

\begin{conjecture}
For any graph $G$ of order $n$, we have
\[
2n \le (n - n^+(G))(n - n^+(G) + 3).
\]
\end{conjecture}

The above conjecture is equivalent to Conjecture \ref{conj:inertia_main} when $n^0 = 0$. It also refines a question of Mohar \cite{Mohar_lecture} on upper-bounding the order of a graph by a function of its non-positive eigenvalues. Even the following weaker problem is of interest.

\begin{problem}\label{problem:polynomial_bound}
    Does there exist a polynomial $f(x)$ such that for every graph $G$,  
    \[n^+(G)\le f(n^-(G)).\]
\end{problem}

\section*{Acknowledgement}

The authors would like to thank Willem Haemers and Shengtong Zhang for helpful comments on Conjecture~\ref{conj:inertia_main}. The authors also thank anonymous referees for their careful reading and helpful comments.

\bibliographystyle{plainurl}
\bibliography{references}

\vspace{0.4cm}

\noindent Saieed Akbari, Email: {\tt s\_akbari@sharif.edu}\\ 
The research visit of S. Akbari at Simon Fraser University was supported in part by the ERC Synergy grant (European Union, ERC, KARST, project number 101071836).\\
\textsc{Department of Mathematical Sciences, Sharif University of Technology, Tehran, Iran}\\[1pt]

\noindent Clive Elphick, Email: {\tt clive.elphick@gmail.com}\\
\textsc{School of Mathematics, University of Birmingham, Birmingham, UK}\\[1pt]

\noindent Hitesh Kumar, Email: {\tt hitesh.kumar.math@gmail.com}, {\tt hitesh\_kumar@sfu.ca}\\
\textsc{Department of Mathematics, Simon Fraser University, Burnaby, BC, Canada}\\[1pt]

\noindent Shivaramakrishna Pragada,\\
Email: {\tt shivaramakrishna\_pragada@sfu.ca, shivaramkratos@gmail.com}\\
\textsc{Department of Mathematics, Simon Fraser University, Burnaby, BC, Canada}\\[1pt]

\noindent Quanyu Tang, Email: {\tt tang\_quanyu@163.com}\\
\textsc{School of Mathematics and Statistics, Xi'an Jiaotong University, Xi'an 710049, P. R. China}

\end{document}